\documentclass[12pt]{article}
\usepackage{latexsym,amsmath,amssymb}

\newcommand{\fl}{\longrightarrow}
\newfont{\bb}{msbm10 at 12pt}

\def\r{\hbox{\bb R}}
\def\d{\hbox{\bb D}}
\def\h{\hbox{\bb H}}

\def\s{\hbox{\bb S}}
\def\mr{\mathbb{M}^2\times\mathbb{R}}
\def\hr{\mathbb{H}^2\times\mathbb{R}}
\def\m{\mathbb{M}^2}
\def\pt{\frac{\partial\ }{\partial t}}
\def\pz{\frac{\partial\ }{\partial z}}
\def\pzb{\frac{\partial\ }{\partial \bar{z}}}

\newcommand{\beq}{\begin{equation}}

\newcommand{\eeq}{\end{equation}}

\newcommand{\norm}[1]{\left\Vert #1 \right\Vert}

\newcommand{\set}[1]{\left\{#1\right\}}
\newcommand{\meta}[2]{\langle #1,#2 \rangle }
\newcommand{\eps}{c(\Sigma)}

\newcommand{\To}{\longrightarrow }

\newcommand{\zb}{\bar{z}}

\usepackage[latin1]{inputenc}
\usepackage {amsmath}
\usepackage {amsthm}
\usepackage{times}
\usepackage{amscd}
\usepackage{epsf}

\numberwithin{equation} {section}

\begin{document}

\theoremstyle{plain}\newtheorem{lema}{Lemma}[section]
\theoremstyle{plain}\newtheorem{pro}{Proposition}[section]
\theoremstyle{plain}\newtheorem{teo}{Theorem}[section]
\theoremstyle{plain}\newtheorem{ex}{Example}[section]
\theoremstyle{plain}\newtheorem{remark}{Remark}[section]
\theoremstyle{plain}\newtheorem{corolario}{Corollary}[section]

\begin{center}
\rule{15cm}{1.5pt} \vspace{.6cm}

{\Large \bf Complete Constant Mean Curvature\\[4mm]
surfaces and Bernstein type Theorems in $\mr$} \vspace{0.4cm}

\vspace{0.5cm}

{\large José M. Espinar$\,^\dag$\footnote{The author is partially supported by
MEC-FEDER, Grant No MTM2007-65249}, Harold Rosenberg$\,^\ddag$}\\
\vspace{0.3cm} \rule{15cm}{1.5pt}
\end{center}

\vspace{.5cm}

\noindent $\mbox{}^\dag$ Departamento de Geometría y Topología, Universidad de
Granada, 18071 Granada, Spain; e-mail: jespinar@ugr.es\vspace{0.2cm}

\noindent $\mbox{}^\ddag$ Institut de Mathématiques, Université Paris VII, 2 place
Jussieu, 75005 Paris, France; e-mail: rosen@math.jussieu.fr

\vspace{.3cm}

\begin{abstract}
In this paper we study  constant mean curvature surfaces $\Sigma$ in a product
space, $\mr$, where $\m$ is a complete Riemannian manifold.  We assume the  angle
function $\nu = \meta{N}{\pt}$ does not change sign on $\Sigma$.  We classify these
surfaces according to the infimum  $c(\Sigma)$ of the Gaussian curvature of the
projection of $\Sigma$.

When $H \neq 0$ and  $c(\Sigma) \geq 0$, then  $\Sigma$ is a cylinder over a
complete curve with curvature $2H$.  If $H=0$  and $c(\Sigma) \geq 0$, then $\Sigma$
must be a vertical plane or $\Sigma$ is a slice $\m \times \set{t}$, or $\m \equiv
\r ^2$ with the flat metric and $\Sigma$ is a tilted plane (after possibly passing
to a covering space).

When $c(\Sigma ) <0 $ and $H > \sqrt{-c(\Sigma )}/2$, then $\Sigma$ is a vertical
cylinder over a complete curve of $\m $ of constant geodesic curvature $2H$.  This result is optimal.

We also prove a non-existence result concerning complete multi-graphs in $\mr$, when $c(\m ) <0$.
\end{abstract}

\section{Introduction}

The image of the Gauss map of a complete minimal surface in $\r ^3 $ may determine
the surface. for example, an entire minimal graph is a plane (Bernsteins' Theorem
\cite{Be}). More generally, if the Gaussian image misses more than four points, then
it is a plane (\cite{Fu}). The Gaussian image of Scherks' doubly periodic surface
misses exactly four points.

The image of the Gauss map of a non-zero constant mean curvature surface in $\r ^3 $
does determine the surface under certain circumstances.  Hoffman, Osserman and
Schoen proved (see \cite{HOS}): {\it Let $\Sigma \subset \r ^3$ be a complete
surface of constant mean curvature. If the image of the Gauss map lies in an open
hemisphere, then $\Sigma$ is a plane. If the image is contained in  a closed
hemisphere, then $\Sigma $ is a plane or a right cylinder.}  Unduloids show that
this result is the best possible.

The aim of this paper is establish results analogous to those of
Hoffman-Osserman-Schoen \cite{HOS} for constant mean curvature surfaces in $\mr$. In
product spaces, the condition that the image of the Gauss map is contained in a
hemisphere, becomes that the angle function, i.e., $\nu =\meta{N}{\pt}$, does not
change sign; here  $N$ denotes a unit normal vector field along a surface $\Sigma
\subset \mr$.

There are many (interesting) complete minimal and constant mean curvature graphs in
$\mr$ (e.g., in $\hr$, \cite{CoR}). We will see that conditions on the value of the
mean curvature $H$, and the curvature of $\m$, can determine complete surfaces in
$\mr$ of constant mean curvature $H$ whose angle function does not change sign.

Let $\Sigma$ be a complete surface of constant mean curvature $H$ immersed in $\mr$.
We will also say $\Sigma$ is an $H-$surface to mean $\Sigma$ has constant mean
curvature $H$. Let $\pi : \Sigma \To \m \equiv \m \times \set{0}$, be the horizontal
projection, and define

$$c(\Sigma )={\rm inf}\set{\kappa ( \pi  (p)) \, : \, p \in \Sigma} $$where $\kappa$
is the Gauss (intrinsic) curvature of $\m$. $c (\m)$ is the infimum of the Gauss
curvature of $\m$.

The main result of this work is the following

\vspace{.4cm}

{\bf Theorem \ref{t3}:}

Let $\Sigma $ be a complete immersed $H-$surface in $\mr$, whose angle function
$\nu$ does not change sign. If $c(\Sigma ) <0 $ and $H > \sqrt{-c(\Sigma )}/2$, then
$\Sigma$ is a vertical cylinder over a complete curve of $\m $ of constant geodesic
curvature $2H$.

\vspace{.4cm}

The proof of Theorem \ref{t3} is inspired by the techniques in \cite{HRS}, where it
is proved that a complete multi-graph in $\hr$, of constant mean curvature $1/2$, is
an entire graph.

Before stating our results and describing the organization of this paper, we discuss
some previous work on this subject.

Entire minimal and constant mean curvature graphs $\mr$ have been studied by several
authors.  When $\m$ is a complete surface with non-negative Gaussian curvature, then
an entire minimal graph in $\mr$ is totally geodesic (\cite{Ro1}). Hence the graph
is a horizontal slice or $\m$ is a flat $\r ^2$ and the graph is a tilted plane.
This result has been generalized to constant mean curvature entire graphs in
(\cite{ADR}).

Entire constant mean curvature $1/2$  graphs, in $\hr$ and entire minimal graphs in
Heisenberg space have been classified (they are sister surfaces, see \cite{DH},
\cite{FM2} and \cite{HRS}).

S. Fornari and J. Ripoll (see \cite{FR}) have considered the general problem of
constant mean curvature hypersurfaces transverse to an ambient Killing field of a
Riemannian manifold, and they obtained several interesting generalizations of the
results of Hoffman, Osserman and Schoen \cite{HOS}. In particular, they prove the
Theorem \ref{t1} stated below under the stronger hypothesis $\m$ has non-negative
curvature.

Now we describe the organization of the paper.

In Section 2, we present the equations that an immersed $H-$surface in $\mr $ must
satisfy.

In Sections 3 we consider the case $c(\Sigma )\geq 0$ and we prove:

\vspace{.4cm}

{\bf Theorem \ref{t1}:}

Let $\Sigma \subset \mr$ be a complete $H-$surface whose angle function does not
change sign. If $ c(\Sigma)\geq 0 $, then
\begin{itemize}
\item  If $H=0$, then $\Sigma$ is a vertical plane or $\Sigma$ is
a slice $\m \times \set{t}$, or $\m \equiv \r ^2$ with the flat metric and $\Sigma$
is a tilted plane (after possibly passing to a  covering space).

\item If $H \neq 0$, $\Sigma$ is a cylinder over a complete curve with curvature $2H$.
\end{itemize}

\vspace{.4cm}

In Section 4, we consider the case $c(\Sigma )<0$. First, we prove non-existence of
certain entire graphs; more precisely:

\vspace{.4cm}

{\bf Lemma \ref{l3}:}

Let $\m$ be a complete surface with  $c(\m)< 0$. Then, there are no entire graphs in
$\mr$, with constant mean curvature $H$, with $H> \sqrt{-c(\m)}/2$.

\vspace{.4cm}

Finally, we prove our main Theorem \ref{t3} that we previously stated.

\section{The geometry of surfaces in $\mr$}

Henceforth $\m $ denotes a complete Riemannian surface with $\partial \m =
\emptyset$. Let $g$ be the metric of $\m $ and $\nabla $ the Levi-Civita connection
on $\mr $ with the product metric $\meta{}{}=g+dt^2$.

Let $\Sigma $ be a complete orientable $H-$surface immersed in $\mr$ and let $N$ be
a unit normal to $\Sigma$. In terms of a conformal parameter $z$ of $\Sigma$, the
first and second fundamental forms are given by
\begin{equation}\label{I}
\begin{array}{l}
I\ =\ \lambda\,|dz|^2\\
II\ =\ p\,dz^2+\lambda\,H\,|dz|^2+\overline{p}\,d\bar{z}^2,
\end{array}
\end{equation}
where $p\,dz^2=\meta{ -\nabla _\pz N}{\pz}\,dz^2$ is the Hopf differential of
$\Sigma$.

Let $\pi :\mr\fl\m $ and $h :\mr\fl\r$ be the usual projections. We also call the
restriction of $h$ to $\Sigma$ the {\it height function}.

First we derive the following necessary equations on $\Sigma$, which were obtained
in \cite{AEG}, but we establish here for the sake of completeness.

\begin{lema}\label{l1}
Given an immersed surface $\Sigma \subset \mr $, the following equations are
satisfied:
\begin{eqnarray}
K(I) &=& K + \kappa \, \nu ^2 \label{gauss}\\
|h_z|^2&=&\frac{1}{4}\, \lambda\,(1-\nu^2)\label{c4}\\
h_{zz}&=&\frac{\lambda_z}{\lambda}\,h_z+p\,\nu\label{c5}\\
h_{z\zb}&=&\frac{1}{2}\,\lambda\,H\,\nu\label{c2}\\
\nu_z&=&-H\,h_z-\frac{2}{\lambda}\,p\,h_{\bar{z}}\label{c3}\\
p_{\zb}&=&\frac{\lambda}{2}\,(H_z+\kappa \,\nu\,h_z)\label{c1}
\end{eqnarray}
where $\kappa (z)$ stands for the Gauss curvature of $\m $ at $\pi (z)$, $K$ the
extrinsic curvature and $K(I)$ the Gauss curvature of $I$.
\end{lema}
\begin{proof}
Let us write
$$\pt=T+\nu\,N$$where $T$ is a tangent vector field on $S$. Since $\pt$ is the gradient in $\m$ of
the function $t$, it follows that $T$ is the gradient of $h$ on $S$.

Thus, from (\ref{I}), one gets $T=\frac{2}{\lambda}(h_{\zb}\pz+h_z\pzb)$ and so
$$
1=\langle\pt,\pt\rangle=\langle T,T\rangle+\nu^2=\frac{4\,|h_z|^2}{ \lambda}+\nu^2,
$$
that is, (\ref{c4}) holds.

On the other hand, from (\ref{I}) we have
\begin{equation}\label{dos}
\begin{split}
\nabla_{\pz}\ \pz &= \frac{\lambda _z}{\lambda}\pz+p\,N\\
\nabla_{\pz}\ \pzb &= \frac{1}{2}\,\lambda\,H\,N\\
-\nabla_{\pz}\ N&= H\,\pz+\frac{2}{\lambda}\,p\,\pzb .
\end{split}
\end{equation}

The scalar product of these equalities with $\pt$ gives us (\ref{c5}), (\ref{c2})
and (\ref{c3}), respectively.

Finally, from (\ref{dos}) we get
$$
\langle\nabla_{\pzb}\nabla_{\pz}\ \pz-\nabla_{\pz}\nabla_{\pzb}\ \pz ,N\rangle\ =\
p_{\zb}-\frac{1}{2}\,\lambda\,H_z.
$$
Hence, using the relationship between the curvature tensors of a product manifold
(see, for instance, \cite[p. 210]{O}), the Codazzi equation becomes
$$ \frac{1}{2}\,\lambda\,\kappa \,\nu\,h_z\ =\
p_{\zb}-\frac{1}{2}\,\lambda\,H_z, $$ that is, (\ref{c1}) holds.

To finish, note that (\ref{gauss}) is nothing but the Gauss equation of the
immersion.
\end{proof}

Now, we will define a quadratic differential depending on $c(\Sigma)$. Denote
$$ Q_{c(\Sigma)} \, dz^2 = (2H \, p - c(\Sigma) h_z ^2)\,dz^2 .$$

Note that this is either the usual Hopf differential (up to the factor $2H$) when
$\m = \r ^2$, or the Abresch-Rosenberg differential when $\m = \h ^2$ or $\m = \s
^2$ (see \cite{AR}).

Now, we will compute the modulus of the gradient and Laplacian of $\nu$.

\begin{lema}\label{l2}
Let $\Sigma$ be an $H-$surface immersed in $ \mr $, with $c(\Sigma) \neq 0$. Then
the following equations are satisfied:
\begin{eqnarray}
\norm{\nabla \nu }^2&=& \frac{c(\Sigma)}{4} ( 4H^2 + c(\Sigma) (1-\nu ^2) -2 K )^2 -
\eps\left( K^2 +\frac{4|Q_{c(\Sigma)}|^2}{\lambda ^2}\right) \label{modnuz}\\
\Delta \nu &=& -\left( 4H^2 +\kappa (1-\nu ^2) - 2K\right)\nu \label{deltanu}
\end{eqnarray}
\end{lema}
\begin{proof}
From (\ref{c3})
$$
|\nu_z|^2\ =\ \frac{4\,|p|^2\,|h_z|^2}{\lambda^2}+
H^2\,|h_z|^2+\frac{2\,H}{\lambda}(p\,h_{\zb}^2+\overline{p}\,h_z^2),
$$
and taking into account that
$$ |Q _{c(\Sigma)}|^2\ =\
4\,H^2\,|p|^2+|h_z|^4-2\,c(\Sigma) \,H(p\,h_{\zb}^2+\overline{p}\,h_z^2),
$$
we obtain, using also (\ref{c4}), that
\begin{equation*}
\begin{split}
|\nu_z|^2&=\left(\frac{|p|^2}
{\lambda}+\frac{H^2\,\lambda}{4}\right)(1-\nu^2)+\frac{c(\Sigma)}{\lambda}\,
\left(4\,H^2\,|p|^2+ \frac{\lambda^2}{16}(1-\nu^2)^2-|Q _{c(\Sigma)}|^2\right)\\
 &= \frac{\lambda }{4}(2H^2 - K)(1-\nu ^2)+c(\Sigma) \frac{\lambda}{4}
 \left( 4\,H^2(H^2 -K)+ \frac{(1-\nu^2)^2}{4}-\frac{4|Q _{c(\Sigma)}|^2}{\lambda ^2}\right)
\end{split}
\end{equation*}where we have used that $4|p|^2= \lambda ^2 (H^2 -K)$. Thus
\begin{equation*}
\begin{split}
\norm{\nabla \nu}^2 &= \frac{4}{\lambda } |\nu _z|^2 = - c(\Sigma) (4H^2 + c(\Sigma)
(1-\nu ^2))K + 2H^2 (1-\nu ^2) \\
 &+ 4 c(\Sigma) H^4 + c(\Sigma) \frac{(1-\nu^2)^2}{4}-c(\Sigma) \frac{4|Q_{c(\Sigma)}|^2}{\lambda ^2}\\
 &=- c(\Sigma) (4H^2 + c(\Sigma) (1-\nu ^2))K + \frac{c(\Sigma)}{4}(4H^2 +c(\Sigma) (1-\nu ^2))^2-
 c(\Sigma) \frac{4|Q_{c(\Sigma)}|^2}{\lambda ^2}\\
 &=\frac{c(\Sigma)}{4} ( 4H^2 + c(\Sigma) (1-\nu ^2) -2 K )^2 - c(\Sigma)
\left( K^2 +\frac{4|Q _{c(\Sigma)}|^2}{\lambda ^2}\right)
\end{split}
\end{equation*}

On the other hand, by differentiating (\ref{c3}) with respect to $\zb$ and using
(\ref{c5}), (\ref{c2}) and (\ref{c1}), one gets
\begin{equation*}
\nu_{z\zb}\ = -\kappa
\,\nu\,|h_z|^2-\,\frac{2}{\lambda}\,|p|^2\,\nu-\,\frac{H^2}{2}\,\lambda\,\nu.
\end{equation*}

Then, from (\ref{c4}),
\begin{equation*}
\begin{split}
\nu_{z\zb}&=-\frac{\lambda\,\nu}{4}\left(\kappa (1-\nu^2)
+\,\frac{8\,|p|^2}{\lambda^2}+2\,H^2\right)\\
 &= -\frac{\lambda}{4}\left( 4H^2 +\kappa (1-\nu ^2) - 2K\right) \nu
\end{split}
\end{equation*}thus
\begin{equation*}
\begin{split}
\Delta \nu &= \frac{4}{\lambda}\nu _{z\zb} =-\left( 4H^2 +\kappa (1-\nu ^2) -
2K\right)\nu
\end{split}
\end{equation*}
\end{proof}

\begin{remark}
Note that (\ref{deltanu}) is nothing but the Jacobi equation for the Jacobi field
$\nu$.
\end{remark}

We say $\Sigma$ is a {\it vertical plane} when $\Sigma = \gamma \times \r$, $\gamma
\subset \m $ a complete geodesic.

\begin{lema}\label{l2.3}
Let $\Sigma $ be a complete $H-$surface immersed in $\mr$ whose angle function $\nu$
is constant. Then,
\begin{itemize}
\item  If $H=0$, $\Sigma$ is a vertical plane or
a slice $\m \times \set{t}$, or $\m \equiv \r ^2$ with the flat metric and $\Sigma$
is a tilted plane (after possibly passing to a  covering space).

\item If $H \neq 0$ when $c(\Sigma) \geq 0$ or $H >1/2$ when $c(\Sigma) =-1$, then
$\Sigma$ is a cylinder over a complete curve with curvature $2H$ in $\m$.
\end{itemize}
\end{lema}
\begin{proof}

We can assume, up to an isometry, that $\nu \leq 0$. We will divide the proof in
three cases:

\begin{itemize}
\item $\nu = 0$:

Using (\ref{c2}), $h$ is harmonic and $\Sigma$ is flat since $\lambda = 4 |h_z|^2$
by (\ref{c4}), thus $\Sigma$ is conformally the plane. So, in this case, $\Sigma $
must be either a vertical plane if $H=0$ or $\Sigma$ is a vertical cylinder over a
complete curve of curvature $2H$ in $\m$.

\item $\nu = -1$:

In this case, $\Sigma $ is a slice, and necessarily $H=0$.

\item $-1 < \nu < 0$:

From (\ref{c3})
$$H h_z = - \frac{2 p}{\lambda} h_{\zb}$$then
$$ H^2 = \frac{4|p|^2}{\lambda ^2} = H^2 -K $$since $|h_z|^2 \neq 0$ from
(\ref{c4}), so $K=0$ on $\Sigma$.

Thus, from (\ref{deltanu}), we have
$$ 4 H^2 + \kappa (1- \nu ^2 ) = 0 .$$

So, if $c(\Sigma) = -1$, this is impossible since $0<1-\nu ^2 <1$ and $4H^2 > 1 = -
c(\Sigma)$.

If $c(\Sigma) \geq 0$, then $H=0$ and $\kappa (\pi ) \equiv 0 $ on $\Sigma$, and we
will show that

{\bf Claim: }$\pi  (\Sigma) = \m$.

It is enough to prove that $\pi  (\Sigma ) $ has no boundary in $\m$. Suppose that
there exists $q \in \partial \pi  (\Sigma ) \subset \m $ and $\set{p_n} \subset
\Sigma$ a sequence such that $\set{\pi (p_n )} \To q $. Fix $p_0 \in \Sigma$ and let
$\gamma _n $ be the complete geodesic in $\Sigma $ joining $p_0$ and $p_n$. Since $q
\in \partial \pi  (\Sigma)$, $\gamma _n$ must become almost vertical at $p_n$ for
$n$ sufficiently large, which means that $N(p_n)$ must become horizontal, but $\nu $
is a constant different from $0$, a contradiction.

Thus $\pi  (\Sigma ) = \m $ and $\m $ is a complete flat surface since $K(I)=0$ by
the Gauss equation, that is, it is isometrically $\r ^2$ (after possibly passing to
a covering space), and $\Sigma $ is a {\it tilted} plane.
\end{itemize}
\end{proof}

We state some basic facts on the theory of eigenvalue problems on Riemannian
manifolds (see \cite{Ch1} and \cite{Ch2} for details). Given a domain $\Omega
\subset \m $ such that $\overline{\Omega}$ is compact and its boundary $\partial
\Omega $ is $C^{\infty }$, the real numbers $\lambda$, called {\it eigenvalues}, are
those for which there exists a nontrivial solution $\phi \in C^2 (\Omega) \cap C^0
(\overline{\Omega})$ to
\begin{equation}\label{dirichlet}
\begin{matrix}
 \overline{\Delta} \phi + \lambda \phi = 0 & \text{ on } \, \Omega \\
 \phi = 0 & \text{ on } \, \partial \Omega
\end{matrix}
\end{equation}where $\overline{\Delta}$ denotes the usual Laplacian operator
associated to the Riemannian metric $g$ on $\Omega$. This problem is called the {\it
Dirichlet eigenvalue problem}.

It is well known that the set of eigenvalues in the Dirichlet problem consists of a
sequence
$$ 0 < \lambda _1 \leq \lambda _2 \leq \ldots \uparrow +\infty ,$$and we will denote
\begin{equation}
\lambda (\Omega) = \lambda _1
\end{equation}to be the lowest eigenvalue in the Dirichlet eigenvalues problem in
$\Omega$.

The last quantity that we will need it is the {\it Cheeger constant}, that is
\begin{equation}\label{cheeger}
\jmath (\m) = {\inf }_{\Omega} \frac{{\rm A}(\partial \Omega)}{{\rm V}(\Omega)}
\end{equation}where $\Omega$ varies over open domains on $\m$ possesing compact
closure and $C^{\infty}$ boundary.

\section{Complete $H-$surfaces $ \Sigma$ in $\mr$ with $c(\Sigma) \geq 0$}

In this Section we shall establish:

\begin{teo}\label{t1}
Let $\Sigma $ be a complete $H-$surface immersed in $\mr$ whose angle function $\nu$
does not change sign. If $ c(\Sigma)\geq 0 $, then
\begin{itemize}
\item  If $H=0$, then $\Sigma$ is a vertical plane or $\Sigma$ is
a slice $\m \times \set{t}$, or $\m \equiv \r ^2$ with the flat metric and $\Sigma$
is a tilted plane (after possibly passing to a  covering space).

\item If $H \neq 0$, $\Sigma$ is a cylinder over a complete curve with curvature $2H$.
\end{itemize}
\end{teo}

\begin{remark}\label{r1}
As we mentioned in the introduction, Theorem \ref{t1} generalizes results in
\cite{ADR} and \cite{FR}. Our proof of Theorem \ref{t1} is inspired by the work of
Hoffman, Osserman and Schoen \cite{HOS}.
\end{remark}

{\it Proof of Theorem \ref{t1}:}

\vspace{.2cm}

Without loss of generality, we can assume that $\Sigma $ is simply-connected and
orientable. Otherwise we  take its universal cover, if $\nu$ is non-positive on the
surface, then it is non-positive on its universal cover. Thus, by the Uniformization
Theorem, we have three possibilities:
\begin{enumerate}
\item[1)] $\Sigma $ is conformally the $2-$sphere:

By (\ref{deltanu}), $\nu$ is a bounded subharmonic function since
$$ 4H^2 +\kappa (1-\nu ^2 ) -2K \geq 2H^2 +2(H^2-K)+ c(\Sigma)(1-\nu ^2)\geq 0 ,$$thus
$\nu$ must be constant since $\Sigma$ is conformally the $2-$sphere. So, from Lemma
\ref{l2.3}, $\Sigma$ is a slice and $\m $ is necessarily compact.

\item[2)] $\Sigma $ is conformally the plane:

By (\ref{deltanu}), $\nu$ is a bounded subharmonic function, then $\nu$ must be
constant ($\Sigma$ is conformally the plane). Thus, again from Lemma \ref{l2.3},
$\Sigma$ is either a vertical cylinder over a curve of curvature $2H$ in $\m$ or
$\Sigma$ is isometrically $\r ^2$ (after possibly passing to a  covering space), and
$\Sigma $ is a {\it tilted} plane.

\item[3)] $\Sigma$ is conformally the disk:

We will show that this case is impossible. Again, by (\ref{deltanu}), $\nu$ is
subharmonic. By the Maximum Principle, if $\nu =0$ at any interior point, then $\nu
$ must vanish identically on $\Sigma$. Thus, using (\ref{c2}), $h$ is harmonic, so
$\Sigma$ is flat since $\lambda = 4 |h_z|^2$ by (\ref{c4}), so $\Sigma$ must be
conformally the plane, which is a contradiction.

Therefore, $-1 \leq \nu < 0$ on $\Sigma $. Now, from (\ref{gauss}) and
(\ref{deltanu}), we have
\begin{equation*}
\begin{split}
\Delta \nu &= -(4H^2 +\kappa (1-\nu ^2 ) -2K)\nu = -(4H^2 +\kappa (1-\nu ^2 )
-2(K(I) - \kappa \nu ^2))\nu\\
 &= +2 K(I) \, \nu - (4H ^2 + \kappa (1 +\nu ^2) )\nu
\end{split}
\end{equation*}thus
\begin{equation}\label{FCS}
\Delta \nu - 2 K(I) \, \nu + (4H ^2 + \kappa (1 +\nu ^2) )\nu =0
\end{equation}so $\nu $ is a strictly negative solution of (\ref{FCS}), but this is
impossible by \cite[Corollary 3 on page 205]{FCS} since in this paper the authors
showed that given a complete metric on the disk $(\d , ds^2)$ there is no positive
(or negative) solution to the equation
$$ \Delta g - a K(I) g + P g =0 \, \text{ on } \, \d ,$$where $\Delta $ is the
Laplacian operator associated to the Riemannian metric $ds ^2$, $K(I)$ the Gauss
curvature of $ds ^2$, $P$ a smooth non-negative function on $\d$ and $a \geq 1$.
\end{enumerate}
\begin{flushright}
$\square $
\end{flushright}

\begin{remark}
Observe that Case 1, that is, when $\Sigma$ is conformally the sphere, could be
obtained in the following (more geometrical) way. If $\Sigma $ is conformally $\s
^2$, then at the highest point of $\Sigma$, one must have $H$ not zero (otherwise
$\Sigma $ is a slice by the Maximum Principle), but then at a lowest point one has
the angle function not constant, but $\nu$ is constant for \eqref{deltanu}, so
$\Sigma$ must be a slice.
\end{remark}

\section{Complete $H-$surfaces $ \Sigma$ with $c(\Sigma) < 0$}

When $c(\Sigma) < 0$, this classification is more complicated to obtain. First, we
will construct a $1-$parameter family of subharmonic functions on $\Sigma$.

\begin{lema}\label{l2}
Let $\Sigma $ be an $H-$surface immersed in $\mr$ whose angle function $\nu$ does
not change sign. Suppose that $c(\Sigma ) =-1$ and $H > 1/2$.

Let us consider the function on $\Sigma$
\begin{equation}\label{funcion}
f _{m}(\nu) = \frac{m}{\sqrt{4\,H^2 -1}}\ {\rm
arcsin}\left(\frac{\nu}{\sqrt{4\,H^2-(1-\nu^2)}}\right) .
\end{equation}

Then for each $m \in \r ^+$, $f_m$ is a subharmonic function on $\Sigma $ such that

$$\frac{m}{\sqrt{4\,H^2-1}}\ {\rm
arcsin}\left(\frac{-1}{2H}\right) \leq f_m (\nu) \leq 0$$on $\Sigma$.
\end{lema}
\begin{proof}

Let us fix $m_0 \in \r ^+$, and consider
\begin{equation}\label{funcion}
f(\nu) = \frac{m_0}{\sqrt{4\,H^2-1}}\ {\rm
arcsin}\left(\frac{\nu}{\sqrt{4\,H^2-1+\nu^2}}\right) ,
\end{equation}thus
$$ \Delta f(\nu) = f'' (\nu)\norm{\nabla \nu}^2 + f'(\nu) \Delta \nu  $$

On the one hand, we have
\begin{eqnarray*}
f'(\nu) &=&\frac{m_0}{4\,H^2 -\,(1-\nu^2)} \\
f''(\nu) &=& \frac{-2 m_0 \nu }{(4\,H^2 -\,(1-\nu^2))^2}
\end{eqnarray*}and, on the other hand
\begin{eqnarray*}
\norm{\nabla \nu }^2&=& -\frac{1}{4} ( 4H^2 - (1-\nu ^2) -2 K )^2 + \left(
K^2 +\frac{4|Q |^2}{\lambda ^2}\right)\\
\Delta \nu &=& -\left( 4H^2 +\kappa (1-\nu ^2) - 2K\right)\nu
\end{eqnarray*}where we have taken $\eps = -1$ and denoted $Q = Q_{\eps}$.

Therefore,
\begin{eqnarray*}
f''(\nu ) \norm{\nabla \nu}^2 &=& \frac{m_0 \nu}{2} - \frac{2 K m_0 \nu}{4H^2
-(1-\nu^2)} + 4 f''(\nu )\frac{|Q|^2}{\lambda ^2}\\
f'(\nu) \Delta \nu &=& - m_0 \frac{4H^2 + \kappa(1-\nu^2)}{4H^2 -(1-\nu^2)} \nu +
\frac{2 K m_0 \nu}{4H^2 -(1-\nu^2)}
\end{eqnarray*}and so
\begin{equation*}
\begin{split}
\Delta f(\nu) &= \frac{m_0 \nu}{2} - m_0 \nu \frac{4H^2 + \kappa(1-\nu^2)}{4H^2
-(1-\nu^2)}  +4 f''(\nu )\frac{|Q|^2}{\lambda ^2} \\
 &= m_0 \nu \left(1 - \frac{4H^2 + \kappa(1-\nu^2)}{4H^2 -(1-\nu^2)}\right) - \frac{m_0
 \nu}{2} +4 f''(\nu )\frac{|Q|^2}{\lambda ^2} .
\end{split}
\end{equation*}

Now, using that $ -1 = c(\Sigma) \leq \kappa $, we have
$$ 1 - \frac{4H^2 + \kappa(1-\nu^2)}{4H^2 -(1-\nu^2)} \leq 0 $$thus
$$ \Delta f(\nu ) \geq 0 $$since $\nu \leq 0 $ and $m_0$ is positive (then $f''(\nu) \geq 0
$).

Moreover, it is easy to see that
$$\frac{m_0}{\sqrt{4\,H^2-1}}\ {\rm
arcsin}\left(\frac{-1}{\sqrt{4\,H^2}}\right)=f(-1) \leq f(\nu) \leq f(0) = 0 $$and
we proved the Lemma \ref{l2}.
\end{proof}

We will need the following

\begin{lema}\label{l3}
There are no entire $H-$graph in $\mr$ with $H> 1/2$ and $c(\Sigma)=-1$.
\end{lema}

\begin{proof}
Let us suppose that such an entire graph exists. Let
$$\Sigma = {\rm Gr}(u)= \set{(x , u(x))\in \mr : \, x \in  \m} ,$$where $u : \m \To \r
$ is a solution of
\begin{equation}\label{divergencia}
\overline{{\rm div }}\left( \frac{\overline{\nabla} u}{\sqrt{1+|\overline{\nabla} u
|^2}}\right)=2H
\end{equation}where $\overline{{\rm div}}$ and $\overline{\nabla}$ denote the divergence and gradient
operators in $\m$.

We will obtain a lower bound for the Cheeger constant of $\m $ in terms of $H$,
following an argument due to Salavessa \cite{Sa}. Let $\Omega \subset \m $ be an
open domain with compact closure and smooth boundary $\partial \Omega$, let us
denote by $\eta$ the outwards normal to $\partial \Omega$. Thus, from
(\ref{divergencia}) and the Divergence Theorem, we have
\begin{equation*}
\begin{split}
2H {\rm V}(\Omega) &= \int _{\Omega} \overline{{\rm div }}\left( \frac{\overline{
\nabla } u}{\sqrt{1+|\overline{\nabla u }|^2}}\right) \, dV = \int _{\partial
\Omega} g\left(\frac{\overline{\nabla } u}{\sqrt{1+|\overline{\nabla} u |^2}} , \eta \right) \, dA \\
 & \leq {\rm A}(\partial \Omega) .
\end{split}
\end{equation*}

Then, an immediate consequence of (\ref{cheeger}) yields
\begin{equation}\label{lowest}
2H \leq \jmath (\m)
\end{equation}

Next, we obtain an upper bound for the infimum, $\lambda (\m)$, of the spectrum of
the Laplacian on $\m$. That is, from \cite{Chg}, for any $p \in \m $ and any $\delta
>0 $ we have
\begin{equation}\label{funtone}
\lambda  (B(p,\delta)) \leq \lambda _{-1} (\delta)
\end{equation}where $\lambda  (B(p,\delta))$ is the lowest eigenvalue
of the Laplacian on the metric ball of radius $\delta$ centered at $p_0$ on $\m$ and
$\lambda _{-1} (\delta)$ the lowest eigenvalue of the Laplacian on the ball of
radius $\delta$ on the space form of constant curvature $-1$.

Now, since $\m $ is complete, letting $\delta \To + \infty$ in \eqref{funtone}, we
have
\begin{equation}\label{bound}
\lambda (\m ) \leq 1 /4
\end{equation}where we have used that
$ \lim _{\delta \To +\infty} \lambda _{-1} (\delta)  = 1/4$ (see \cite[Theorem 5,
pag 46]{Ch1}).

Finally, Cheeger's inequality (see \cite[Theorem VI.1.2, pag 161]{Ch2})
$$ \jmath (\m)^2 \leq 4 \lambda (\m)  $$combined with (\ref{bound}) give us
\begin{equation}\label{upper}
\jmath (\m) \leq 1 = \sqrt{-c(\m )} .
\end{equation}

Thus, from (\ref{lowest}) and (\ref{upper}),
\begin{equation*}
1 < 2H \leq \jmath (\m) \leq 1
\end{equation*}which is a contradiction and we have proved the Lemma \ref{l3}.
\end{proof}

In fact, Lemma \ref{l3} can be generalized as follows.

\begin{corolario}\label{cor1}
Let $\m$ be a surface with  $c(\m)=-1 $. Then, there is no complete entire vertical
graph in $\mr$ with $${\rm inf} \, H > 1/2 .$$
\end{corolario}
\begin{proof}
With the notation of Lemma \ref{l3}, the only change is
\begin{equation*}
\begin{split}
2 \, {\inf } \,H  \, {\rm V}(\Omega) &\leq \int _{\Omega} H \, dV =\int _{\Omega}
\overline{{\rm div }}\left( \frac{\overline{\nabla } u}{\sqrt{1+|\overline{\nabla } u |^2}}\right) \, dV  \\
 & = \int _{\partial \Omega} g\left(\frac{\overline{\nabla } u}{\sqrt{1+|\overline{\nabla } u |^2}} , \eta \right)
 \, dA\leq {\rm A}(\partial \Omega)
\end{split}
\end{equation*}so,
\begin{equation}\label{lowest2}
2 \, {\inf } \,H \leq \jmath (\m)
\end{equation}and we conclude as in the previous result.
\end{proof}

Now, we establish our main result.

\begin{teo}\label{t3}
Let $\Sigma $ be a complete immersed $H-$surface in $\mr$, whose angle function
$\nu$ does not change sign. If $c(\Sigma ) <0 $ and $H > \sqrt{-c(\Sigma )}/2$, then
$\Sigma$ is a vertical cylinder over a complete curve of $\m $ of constant geodesic
curvature $2H$.
\end{teo}
\begin{proof}

We divide the proof in two steps. First, we will prove that either the surface is a
cylinder over a complete curve or is a multi-graph. Second, we will prove that such
a multi-graph can not exist.

\vspace{.3cm}

\textbf{Claim A:} $\Sigma$ is either a cylinder over a complete curve with curvature
$2H$, or a multi-graph conformally equivalent to the disk.

\vspace{.1cm}

\emph{Proof of Claim A:} By passing to the universal covering space, we can assume
$\Sigma $ is simply connected. Thus,
\begin{enumerate}
\item[1)] Suppose $\Sigma $ is conformally a $2-$sphere and $H  > 1/2$; we will see that this case is
impossible.

From Lemma \ref{l2}, $f_{m}(\nu) $ is a subharmonic and bounded function on $\Sigma$
which is conformally the sphere, thus it must be constant, and therefore, $\nu$ must
be constant on $\Sigma$. Thus, by Lemma \ref{l2.3}, this is impossible.

\item[2)] Suppose $\Sigma $ is conformally the plane and $H  > 1/2$; we will see that
$\Sigma$ must be a vertical cylinder.

From Lemma \ref{l2}, $f_{m}(\nu) $ is a subharmonic and bounded function on $\Sigma$
which is conformally the plane, thus it must be constant, and therefore, $\nu$ must
be constant on $\Sigma$. Thus, by Lemma \ref{l2.3}, $\Sigma$ is a vertical cylinder
over a curve of curvature $2H$ in $\m$.

\item[3)] Next suppose $\Sigma$ is conformally the disk. We will show that
in this case $\Sigma$ is a multi-graph. Moreover, if $H  > 1/\sqrt{3}$; we will show
that this case is impossible.

First, we will show that

{\bf Claim:} $\nu $ must be negative on $\Sigma$.

By (\ref{funcion}), $f \equiv f_{m_0}$, for $m_0 >0$, is subharmonic. Thus, by the
Maximum Principle, if $\nu = 0$ at any interior point, then $\nu$ must vanish
identically on $\Sigma$. This means that $\Sigma $ must be flat, so conformally the
plane, which is a contradiction. Thus, $\nu < 0$ on $\Sigma$ and the Claim is
proved.

Therefore, $-1 \leq \nu < 0$ on $\Sigma $. Now, from (\ref{gauss}) and
(\ref{deltanu}), we have
\begin{equation*}
\begin{split}
\Delta \nu &= -(4H^2 +\kappa (1-\nu ^2 ) -2K)\nu = -(3H^2 +\kappa +H^2 -K  -K(I))\nu\\
 &= +K(I) \, \nu - ( (3H ^2 + \kappa )  +(H^2 -K) )\nu
\end{split}
\end{equation*}thus
\begin{equation}\label{FCS2}
\Delta \nu -  K(I) \, \nu + P \nu =0 .
\end{equation}

Hence $\nu $ is a strictly negative solution of (\ref{FCS2}), and
$$ P = (3H ^2 + \kappa )  +(H^2 -K) \geq 0 .$$

By the assumption that $H> 1/\sqrt{3}$, this is impossible by \cite[Corollary 3 on
page 205]{FCS}.
\end{enumerate}

So, Claim A is proved.
\begin{flushright}
$\blacksquare $
\end{flushright}

Now, we continue with the second step:

\vspace{.3cm}

\textbf{Claim B:} $\Sigma $ can not be a multi-graph.

\vspace{.1cm}

\emph{Proof of Claim B:} We know that $\Sigma$ can not be an entire graph by Lemma
\ref{l3}. Thus the proof will be completed when we prove that such a multi-graph is
in fact an entire graph. The proof of this will be rather long. The idea originates
in the paper \cite{HRS}, where it is proved that a complete multi-graph in $\hr $,
with $H = 1/2$, is in fact an entire graph.

Let us remark that there is a simple geometrical argument to see that there are no
entire vertical graphs with CMC $H>1/2$ in $\hr$, and this fact is as follows: one
could touch such an entire graph by a compact rotational $H-$sphere (touch on the
mean convex side of the graph), and the Maximum Principle would say that $\Sigma$ is
equal to the sphere, a contradiction.

Now, we will show that $\Sigma$ is an entire graph, assuming $\pt $ is transverse to
$\Sigma$:

Since $H > 1/2$, the mean curvature vector of $\Sigma$ never vanishes, so $\Sigma$
is orientable. Let $N$ denote a unit normal field to $\Sigma$. Since $\nu$ is a
non-zero Jacobi function on $\Sigma$ (see Remark 2.1), $\Sigma$ is strongly stable
and thus has bounded curvature. We assume $\nu <0$ and
$\meta{N}{\overrightarrow{H}}>0$.

As $\Sigma$ has bounded geometry, there exists $\delta > 0$ such that for each $p
\in \Sigma$, $\Sigma$ is a graph in exponential coordinates over the disk
$D_{\delta} \subset T_p \Sigma$ of radius $\delta$, centered at the origin of $T_p
\Sigma$. This graph, denoted by $G(p)$, has bounded geometry. $\delta$ is
independant of $p$ and the bound on the geometry of $G(p)$ is uniform as well.

We denote by $F(p)$ the surface $G(p)$ translated to height zero $\m \equiv \m
\times \set{0}$, i.e, let $\phi _p$ be the isometry of $\mr$ which takes $p $ to
$\pi (p)$, then $F(p) = \phi _p ( G(p) )$.

For $q \in \mr$, we will denote by $C_{\delta} (q)$ an arc of $\m $ with geodesic
curvature $2H$, of length $2 \delta$ and centered at $q$, i.e, $q \in C_{\delta}(q)
$ is the mid-point.

\vspace{.3cm}

{\bf Claim 1:} Let $\set{p_n} \in \Sigma $, satisfy $\nu (p_n) \To 0 $ as $n \To +
\infty$, that is, $T_{p_n} \Sigma$ are becoming vertical. Let $\pi (p_n)=q_n$, and
assume $q_n$ converges to some point $q$.  Then, there is a subsequence of $\set{
p_n }$ (which we also denote by $\set{p_n}$) such that $F(p_n)$ converges to
$C_{\delta} (q)\times [-\delta , \delta]$, for some $2H$ arc $C_{\delta}(q)$ at $q$.
The convergence is in the $C^2 -$topology.

{\it Proof of Claim 1:} The proof is the same as in \cite[Claim 1]{HRS} replacing
horocycles by arcs of curvature $2H$.

\vspace{.3cm}

Now, let $p \in \Sigma$ and assume $\Sigma$ in a neighborhood of $p$ is a vertical
graph of a function $f$ defined on $B_R$,  $B_R$ the open geodesic ball of radius
$R$ of $\m $ centered at $\pi(p) = O \in \m$. Denote by $S(R)$ the graph of $f$ over
$B_R$. If $\Sigma$ is not an entire graph then we let $R$ be the largest such $R$ so
that $f$ exists. Since $\Sigma$ has constant mean curvature, $f$ has bounded
gradient on relatively compact subsets of $B_R$.

Let $q \in \partial B_R $ be such that $f$ does not extend to any neighborhood of
$q$ to an $H> 1/2 $ graph.

\vspace{.3cm}

{\bf Claim 2:} For any sequence $q_n \in B_R$, converging to $q$, the tangent planes
$T_{p_n}S(R)$, where $p_n = (q_n , f(q_n))$, converge to a vertical plane $P$. $P$
is tangent to $\partial B_R$ at $q$ (after   translation of $T_{p_n}S(R)$ to height
zero in $\mr$).

{\it Proof of Claim 2:} The same proof as in \cite[Claim 2]{HRS}.

\vspace{.3cm}

Now, from Claim 1 and Claim 2, we know that for any sequence $q_n \in B_R$
converging to $q$, the $F(q_n)$ converge to $C_{\delta} (q) \times [-\delta ,
\delta]$.

\vspace{.3cm}

{\bf Claim 3:} For any $q_n \to q$, $q_n \in B_R$, we have $f(q_n) \To + \infty $ or
$f(q_n) \To - \infty$.

{\it Proof of Claim 3:} Let $\gamma$ be a compact horizontal geodesic of length
$\varepsilon $ starting at $q$, entering $B_R$ at $q$, and orthogonal to $\partial
B_R$ at $q$. Let $\Gamma $ be the graph of $f$ over $\gamma$. Notice that $\Gamma $
has no horizontal tangents at points near $q$ since the tangent planes of $S(R)$ are
converging to $P$. So assume $f$ is increasing along $\gamma $ as one converges to
$q$. If $f$ were bounded, then $\Gamma $ would have a finite limit point $(q, c)$
and $\Gamma$ would have finite length up till $(q,c)$. Since $\Sigma $ is complete,
$(q,c)\in \Sigma$. But then $\Sigma$ would have a vertical tangent plane at $(q,c)$,
a contradiction. This proves Claim 3.

\vspace{.3cm}

Now choose $q_n \in \gamma $, $q_n \To q$, and $F(p_n)$ converges to $C_{\delta}(q)
\times [-\delta , \delta]$. Let $C$ be the complete curve of $\m$ with $q \in C$ and
geodesic curvature $2H$, such that $C$ contains $C_{\delta}(q)$, and parametrize $C$
by arc length; denote $q(s) \in C$ the point at distance $s$ on $C$ from $q(0)=q$,
$s \in \r$. Note that $C$ may have self-intersections, and may be compact and
smooth. Denote by $\gamma (s)$ a horizontal geodesic arc orthogonal to $C$ at
$q(s)$, $q(s)$ is the mid-point of $\gamma (s)$. Assume the length of each $\gamma
(s)$ is $2 \varepsilon$ and
$$\bigcup _{s \in \r } \gamma (s) = T_{\varepsilon } (C)$$is the $\varepsilon -$tubular
neighborhood of $C$.

Let $\gamma ^+ (s)$ be the part of $\gamma (s)$ on the mean concave side of $C$; so
$\gamma = \gamma ^+(0)$. More precisely, the mean curvature vector of $\Sigma $
points down in $\mr $, and $f \to + \infty$ as one approaches $a$ along $\gamma $,
so $C$ is concave towards $B_R$; i.e., $B_R$ is on the concave side of
$C_{\delta}(q)$ at $q$.

\vspace{.3cm}

{\bf Claim 4:} For $n$ large, each $F(q_n)$ is disjoint from $C \times \r$. Also,
for $|s| \leq \delta$, $F(q_n) \cap \gamma (s)$ is a vertical graph over an interval
of $\gamma  (s)$.

{\it Proof of Claim 4:} Choose $n_0$ so that for $n \geq n_0$, $\Gamma _n (s) =
F(q_n) \cap (\gamma (s) \times \r)$ is one connected curve of transverse
intersection, for each $s \in [-\delta , \delta]$. Since the $F(q_n)$ are $C^2
-$close to $C_{\delta}(q) \times [-\delta , \delta]$, $\Gamma _n (s)$ has no
horizontal or vertical tangents and is a graph over an interval in $\gamma (s)$.

We now show that this interval is in $\gamma ^+(s) -q(s)$. Suppose not, so $\Gamma
_n (s)$ goes beyond $C \times \r $ on the convex side. Recall that $\Gamma = \gamma
\cap P^{\bot}$ is the graph of $f$ and $f \To + \infty$ as one goes up on $\Gamma$.
We have $p_n = (q_n , f(q_n))$. Fix $n \geq n_0$ and choose new points $q_k$, $k
\geq n$, so that $f(q_{k+1})- f(q_k) = \delta $; clearly $q_k \To q $ as $k \To +
\infty $. Lift each $\Gamma _k (s)$ to $G(p_k)$ by the vertical translation of
$\Gamma _k (s)$ by $f(q_k)$. The curve $\Gamma (s) = \bigcup _{k \geq n} \Gamma _k
(s)$ is a vertical graph over an interval in $\gamma (s)$. It has points in the
convex side of $C \times \r$ for some $s_0 \in [-\delta , \delta ]$. For $s =0$,
$\Gamma (0) = \Gamma $ stays on the concave side of $C \times \r$. So, for some
$s_1$, $0 < s_1 \leq s_0 $, $\Gamma (s_1)$ has a point on $C \times \r $ and also
inside the convex side of $C \times \r $.

But the $F(q_k)$ converge uniformly to $\Gamma _{\delta }(q) \times [-\delta ,
\delta]$ as $k \To + \infty$, so the curve $\Gamma (s_1)$ converges to $q(s_1)
\times \r $ as the height goes to $+\infty$. This obliges $\Gamma (s_1)$ to have a
vertical tangent on the convex side of $C \times \r$, a contradiction. This proves
Claim 4.

\vspace{.3cm}

Now we choose $\varepsilon _1 < \varepsilon $ (which we call $\varepsilon$ as well)
so that $\bigcup _{s \in [-\delta , \delta]} \Gamma (s)$ is a vertical graph of a
function $g$ on $\bigcup _{s \in [-\delta , \delta ]}(\gamma ^+ (s) - q(s))$, (the
$\gamma ^+ (s)$ have length $\varepsilon _1$).

Before we continue, note that until here, the proof of Theorem \ref{t3} is the same
proof as in \cite[Theorem 1.2]{HRS} with slight modifications. Now the proof
continues differently.

The graph of $g$ converges to $C_{\delta }(q) \times \r$ as the height goes to
infinity.

Now we begin this process again, replacing $\Gamma$ by the curve $\Gamma (\delta)$.
This analytically continues the graph $g$ to a graph over $ \bigcup _{s \in [-
\delta , 2 \delta]} (\gamma ^+ (s) - q(s))$ which converges uniformly to $C(q , [-
\delta , 2 \delta]) \times \r $ as the height goes to infinity. Here $C(q, [-\delta,
2 \delta])$ denotes the arc of $C$, of length $3 \delta$, between the points
$q(-\delta )$ and $q(2 \delta)$. We now continue analytically, by extending the
graph about $C(2 \delta)$. When we refer here to analytic continuation, we mean the
unique continuation of the local pieces of the surface.

We want to extend so that the surface we obtain is within the $\varepsilon -$tubular
neighborhood, $T_{\varepsilon}(C \times \r) \equiv T_{\varepsilon} (C) \times \r $,
of $C \times \r$. Note that the graph of $g$ on $ \bigcup _{s \in [-\delta , \delta
]}(\gamma ^+ (s) - q(s)) $ has this property.

To do this, we go up high enough on $\Gamma (\delta)$ (and $\Gamma (-\delta)$), to
height $t_1$ say, so that all the curves $\Gamma (s)$ starting at height $t_1 $, for
$s \in [\delta , 2 \delta ] \cup [-2 \delta , -\delta ]$, are $\varepsilon -$close
to $C \times \r $.

Now continue this process replacing $\Gamma (\delta)$ and $\Gamma (-\delta)$ by
$\Gamma (2\delta)$ and $\Gamma (-2 \delta)$; again, going up high enough on these
curves so that the graph, possibly immersed if $C$ is not embedded, is within
$T_{\varepsilon}(C \times \r )$.

Let $M$ denote the surface obtained by this analytic continuation, $M$ is a union of
curves $\Gamma (s)$, starting at different heights, each $\Gamma (s)$ is a graph
over $\gamma ^+ (s)$, converging uniformly to $q(s)\times R ^+$.

Let $T^+ _{\varepsilon }(\widetilde{C}\times \r)$ be the universal covering space of
$$ \bigcup _{s \in \r} (\gamma ^+ (s) \times \r) ,$$and we recall that
each $\gamma ^+ (s)$ is the geodesic of length $\varepsilon $ starting at $q(s)$,
orthogonal to $C $, and going to the side of $C$ where $M$ was constructed.

Now, $\widetilde{C}$ is diffeomorphic to $\r$ and $\widetilde{C} \To C $ is the
immersion of $C$ in $\m $. $T ^+ _{\varepsilon }(\widetilde{C}\times \r)$ is an
$\varepsilon -$(one-sided) tubular neighborhood of $\widetilde{C} \times \r$. We
give $T^+ _{\varepsilon }(\widetilde{C}\times \r)$ the metric induced by that of
$\mr $. We lift $M$ to an $H-$surface $\widetilde{M} \subset T^+ _{\varepsilon
}(\widetilde{C}\times \r)$, $\widetilde{M}$ is asymptotic to $\widetilde{C} \times
\r$ as the height goes to infinity. Let $\beta = \partial \widetilde{M}$.

\vspace{.3cm}

{\bf Claim 5:} The surface
$$ Q = \widetilde{C} \times \r $$is an unstable $H-$surface.

{\it Proof of Claim 5:} The stability operator $J$ of $Q$ is
$$ \Delta + |A|^2 + {\rm Ric}(\overrightarrow{n}) ,$$where $\Delta $ is the Laplacian
operator associated to the Riemannian metric induced on $Q$, $| A |^2$ is the square
of the norm of the shape operator associated to $Q$ and ${\rm
Ric}(\overrightarrow{n})$ is the Ricci curvature in the direction of the unit normal
vector field $\overrightarrow{n}$  along $Q$. Since $Q$ is part of a vertical
cylinder, the extrinsic curvature vanishes identically on $Q$ and the unit normal
$\overrightarrow{n}$ is horizontal, hence
$$ J \equiv \Delta + a ,$$where $a = 4H^2 + \kappa \geq 4 H^2 - 1 >0 $.

Consider the operator $J$ on $[0, L] \times [-r , r]$ for $r>0$, where $[0, L]$ is
an interval  of length $L$ on $ \widetilde{C}$.

It is well known that
$$\varphi _1 (t) = \cos \left(\frac{\pi t}{L^2}\right)$$is a first eigenfunction of
$\Delta $ on $[0,L]$, with eigenvalue $\lambda _1 = \frac{\pi ^2}{L^2}$. Similarly,
a first eigenfunction $\varphi _2$ of $\Delta$ on $[-r, r]$ is
$$ \varphi _2 (t) = \cos\left(  \frac{\pi t}{2 r}\right) $$with eigenvalue $\lambda _2 = \frac{\pi ^2}{4
r^2}$.

Let $\varphi = \varphi _1 \times \varphi _2$, so that
$$ \Delta \varphi + (\lambda _1 + \lambda _2) \varphi = 0 \text{ , on } [0,L]\times [-r , r] ,$$then,
$$ J \varphi + (\lambda _1 + \lambda _2 - a ) \varphi = 0 \text{ , on } [0,L]\times [-r , r]
.$$

Hence, if $r$ and $L$ satisfy
$$ \lambda _1 + \lambda _2 - a < 0 , $$then the domain is unstable.

This condition is
\begin{equation}\label{cond}
\frac{\pi ^2}{L^2} + \frac{\pi ^2}{ 4 r^2 } < 4H^2 +\kappa  ,
\end{equation}but for $L$ and $r$ large enough (note that we identify $\widetilde{C}$
with $\r $ and we can choose $L$ large), it is clear that
$$ \frac{\pi ^2}{L^2} + \frac{\pi ^2}{ 4 r^2 } < 4H^2 -1 ,$$ so condition \eqref{cond} is
fulfilled. And the Claim 5 is proved.

Start with a compact stable domain $K_0$ of $Q$. Let $K_0$ expand until one reaches
a stable-unstable domain $K$ of $Q$, $K$ compact. This means there is a smooth
function $f: K \To \r$, $f=0$ on $\partial K$, $f>0 $ on ${\rm int} K$, and $f$
satisfies
$$ J f + \lambda f = 0 , \, \lambda <0 .$$

Let $K(t)$ be the variation of $K$ given by
$$ K(t) = {\rm exp}_p (p + t f(p)N(p)) ,$$where $p\in K$ and $N(p)$ is a unit normal
to $K$. $K(t)$ is a smooth surface with $\partial K(t) = \partial K \subset Q$, and
for $t$ small, $t\neq 0$, ${\rm int}K(t)\cap Q = \emptyset$.

Since the linearized operator $J$ is the first variation of the mean curvature at
$t=0$, and $Jf (p)= - \lambda f(p) >0 $ for $p \in {\rm int}K$, we conclude
$H(K(t))> H$ for $t>0$, and $H(K(t))< H$ for $t<0$.

Now on any compact set of $Q$, $\beta $ is a positive distance from $Q$. So for $t$
small enough the surfaces $K(t)$ are disjoint from $\beta$ and they can be slid up
and down $Q$ to remain disjoint from $\beta$. But $M$ is asymptotic to $Q$ so for
small $t>0$, the surface $K(t)$ will touch $M$ at a first point, when $K(t)$ is slid
up or down $Q$. But this contradicts the Maximum Principle: If $M(1)$ is on the mean
convex side of $M(2)$ near $p$, then the mean curvature of $M(1)$ at $p$ is greater
than or equal to the mean curvature of  $M(2)$ at p. So, Claim A is proved.
\begin{flushright}
$\blacksquare $
\end{flushright}

This completes the proof of Theorem \ref{t3}.
\end{proof}

\begin{remark}
In \cite{Ro}, the second author proved the following result

\vspace{.3cm}

{\bf Theorem A:} Let $N^3$ be a complete riemannian manifold, and suppose, $H$ and
$C>0$ are constants satisfying
$$ 3 H^2 + S \geq C , $$where $S$ is the scalar curvature function of $N^3 $.
Then, if $\Sigma $ is a complete stable $H-$surface, one has
$$ d_{\Sigma} \left( p , \partial \Sigma \right) \leq \frac{2 \pi}{ \sqrt{3C}} .$$

When $\partial \Sigma = \emptyset$, $\Sigma $ is topologically the sphere $\s ^2$.

\vspace{.3cm}

The proof of Theorem A involves studying the metric $d \tilde{s} ^2 = u ds ^2 $,
where $ds ^2$ is the induced metric on $\Sigma$, and $u$ is a positive solution of
the linearized operator
$$ L = \Delta + |A|^2 + {\rm Ric}(n), $$where $\Delta $ is the Laplacian operator
associated to $ds^2$, $|A|^2$ is the square of the norm of the second fundamental
form, and ${\rm Ric}$ is the Ricci curvature in the direction of the normal vector
field, $n$, along $\Sigma$. This metric $d \tilde{s} ^2$ has positive curvature, so when $\partial
\Sigma = \emptyset$, $\Sigma $ is a compact sphere.

This Theorem A gives some insight into Theorems \ref{t1} and \ref{t3} (Theorem
\ref{t1} when $H>0$ and  Claim A of Theorem \ref{t3} when $H > 1/ \sqrt{3}$) since
in the case $3H^2 + S \geq C > 0$, $\Sigma $ is a multi-graph implies $\Sigma $ is a
sphere. This is ruled out by looking at a highest and a lowest point of $\Sigma$.
\end{remark}

\end{document}